\documentclass{article}
\usepackage[utf8]{inputenc}

\usepackage[a4paper]{geometry}
\usepackage{amssymb}
\usepackage{float}
\renewcommand{\geq}{\geqslant}
\renewcommand{\leq}{\leqslant}
\renewcommand{\ge}{\geqslant}
\renewcommand{\le}{\leqslant}

\title{\bfseries\Large\textsc{ Spiders and their Kin: An Investigation of Stanley's Chromatic Symmetric Function for Spiders and Related Graphs}}
\author{
\begin{tabular}{ccc}
\begin{tabular}[t]{c}
Ang\`ele M. Foley \thanks{Formerly Ang\`ele M. Hamel}\\
\small Department of Physics and Computer Science\\
\small Wilfrid Laurier University\\
\small Waterloo, Ontario, Canada
 \end{tabular}  
&
\begin{tabular}[t]{c}
Joshua Kazdan \thanks{Undergraduate student}\\
\small Stanford University\\
\small Stanford, California, USA
\end{tabular}  
&
\begin{tabular}[t]{c}
Larissa Kr\"{o}ll  \footnotemark[2]\\
\small University of Innsbruck \\
\small Innsbruck, Austria\\
\end{tabular}  
\\ \addlinespace[2ex]
\begin{tabular}[t]{c}
Sof\'{i}a Mart\'{i}nez Alberga \footnotemark[2]\\ 
\small University of California, Riverside\\
\small Riverside, California, USA
\end{tabular}
&
\begin{tabular}[t]{c}
Oleksii Melnyk \footnotemark[2]\\
\small University of Oxford\\
\small Oxford, UK
\end{tabular}
&
\begin{tabular}[t]{c}
Alexander Tenenbaum \footnotemark[2]\\
\small University of Toronto\\
\small Toronto, Canada
\end{tabular}
\end{tabular}
}

\usepackage{comment}
\usepackage{ytableau}
\usepackage{subcaption}
\usepackage{hyperref}
\usepackage{graphicx}
\usepackage{color}
\usepackage{amsmath}
\usepackage{mdframed}
\usepackage[export]{adjustbox}
\usepackage{graphicx}					
\usepackage{booktabs}					
\usepackage{hyperref}
\usepackage{multicol}
\usepackage{amsmath}
\usepackage{amsfonts}
\usepackage{amsthm}
\newtheorem{theorem}{Theorem}
\newtheorem{prop}{Proposition}

\newtheorem{lemma}{Lemma}
\newtheorem{df}{Definition}
\newtheorem{remark}{Remark}
\newtheorem{conj}{Conjecture}

\newcommand{\nat}{P(\mathbf{m})}

\begin{document}

\maketitle

\begin{abstract}
We study the chromatic symmetric functions of graph classes related to spiders, namely generalized spider graphs (line graphs of spiders), and what we call horseshoe crab graphs.
We show that no two generalized spiders have the same chromatic symmetric function, thereby extending the work of Martin, Morin and Wagner.  Additionally, we establish that a subclass of generalized spiders, which we call generalized nets, has no $e$-positive members, providing a more general counterexample to the necessity of the claw-free condition.
We use yet another class of generalized spiders to construct a counterexample to a problem involving the $e$-positivity of claw-free, $P_4$-sparse graphs, showing that Tsujie's result on the $e$-positivity of claw-free, $P_4$-free graphs cannot be extended to graphs in this set.  Finally, we investigate the $e$-positivity of another type of graphs, the horseshoe crab graphs (a class of unit interval graphs), and prove the positivity of all but one of the coefficients.  This has close connections to the work of Gebhard and Sagan and Cho and Huh.
\end{abstract}

\section{Introduction}\label{intro}
The chromatic symmetric function, defined by Richard Stanley \cite{Stan} in 1995, is a graph invariant that generalizes the chromatic polynomial.  In order to define the chromatic symmetric function, we consider all proper colorings of a graph.  A \textbf{coloring} of a graph $G$ is a function, \[\kappa: V \rightarrow \mathbb{N}.\] 
Such a coloring is considered \textbf{proper} if $\kappa(u) \neq \kappa(v)$, where $u$ and $v$ are vertices connected by an edge.
If $G$ has a vertex set $V= \{v_1, v_2, v_3, \ldots, v_n \}$, then the \textbf{chromatic symmetric function} of $G$  is defined as the sum
\begin{equation*}
X_{G}= \sum_{\kappa}  x_{\kappa(v_1)}  x_{\kappa(v_2)}  x_{\kappa(v_3)} \cdots x_{\kappa(v_n)},
\end{equation*}
over all proper colorings $\kappa$.

The chromatic symmetric function provides information about the structure of a graph, including the number of vertices, edges, and possible acyclic orientations. 
 In 1993 Stanley and Stembridge \cite{StanStem} conjectured that the chromatic symmetric functions of claw-free incomparability graphs can be written as a linear combination of elements in the elementary symmetric function basis with nonnegative coefficients, a property called $e$-{\em positivity}. This conjecture prompted investigation into classes of $e$-positive graphs (by an abuse of notation, we say a graph is $e$-positive if its chromatic symmetric function is $e$-positive). Stanley \cite{Stan} conjectured in 1995 that the chromatic symmetric function distinguishes non-isomorphic \textbf{trees}, i.e. connected graphs without cycles. Since then Gasharov \cite{Gash} proved that claw-free incomparability graphs are Schur-positive. Regarding the stronger condition of $e$-positivity, Gebhard and Sagan \cite{GebSag} and Dahlberg and van Willigenburg \cite{DahlWill} have proved that $k$-chains and lollipop graphs are $e$-positive. Subsequently, Dahlberg, Foley, and van Willigenburg \cite{DahlFolWill} gave three infinite classes of graphs that are not $e$-positive. Furthermore, by a reduction of Guay-Paquet \cite{GuayPaq}, if all claw-free, unit-inverval incomparability graphs have $e$-positive chromatic symmetric functions, then Stanley's $e$-positivity conjecture holds for all graphs. 
 
 Regarding uniqueness, Martin, Morin, and Wagner \cite{MarMorWag} have established that particular trees (that is, spiders and some caterpillars) are uniquely determined by their chromatic symmetric functions. Tsujie \cite{Tsuj} also proved that trivially perfect graphs are uniquely determined by their chromatic symmetric functions. 
 
 The quasisymmetric variant of the chromatic symmetric function was introduced by Shareshian and Wachs \cite{SharWa}, which led to a generalization of some of Stanley's conjectures to chromatic quasisymmetric functions. For example Cho and Huh \cite{ChoHuh} used injections in order to prove $e$-positivity of certain classes of natural unit interval graphs. Cho and Hong \cite{ChoHong} extend the techniques of \cite{ChoHuh} to prove the Stanley-Stembridge conjecture for the class of unit interval graphs with independence number 3. Harada and Precup \cite{HarPre} have connected these functions to Hessenberg varieties and Pawlowski \cite{Paw} has studied chromatic symmetric functions through the group algebra of $S_n$.  
 
 In order to establish general properties of chromatic symmetric functions, researchers have typically considered more manageable families of graphs, e.g.\ spiders, caterpillars, squids, and crabs \cite{MarMorWag};  lollipops and lariats \cite{DahlWill}.  In this spirit we explore  generalized spiders and horseshoe crabs (no relation to the crabs of \cite{MarMorWag}).  In biology spiders and horseshoe crabs belong to the same phylum, {\em arthropoda}, thus there is a nice affinity between them. In graph theoretic terms these amount to $K_n$ ``bodies" with pendant vertices, attached paths, or other types of protrusion. In this way these graphs generalize both the lollipops ($K_n$ with one path attached) and nets ($K_3$ with three pendant vertices).
 
We begin by outlining some basic definitions and formulas from graph and symmetric function theory in Section \ref{back}. In Section \ref{Nets}
 we generalize Stanley's basic example of a claw-free, non-$e$-positive graph, known as the net (see Figure \ref{net}), to an infinite family of graphs that are not $e$-positive. We also use this result to show that the class of claw-free, $P_4$-free graphs, which have been proven to be $e$-positive by Tsujie in \cite{Tsuj}, cannot be further extended to claw-free, $P_4$-sparse graphs. In Section \ref{Uniq} we consider Stanley's second conjecture, demonstrating that the chromatic symmetric function distinguishes the non-isomorphic generalized spiders from each other, suggesting that the chromatic symmetric function could distinguish classes of graphs that are not trees.  Finally in Section \ref{genLol} we examine the chromatic quasisymmetric function of natural unit interval graphs. We consider the $e$-positivity  of a  certain class of graphs, the horseshoe crab graphs. Our approach generalizes the method of weight preserving injections introduced by Cho and Huh. Section \ref{future} suggests possible avenues for future research.

\section{Background and Notation}\label{back}

We will start by defining some of the additional notation necessary for this paper.

Given a graph $G=(V,E)$, we define the \textbf{line graph} $H=(V', E')$ as follows.  Let $|V'|=|E|$, and say that $\varphi: V'\rightarrow E$ is a bijection between these sets.  For $v_1, v_2\in V'$, $(v_1, v_2)\in E'$ if and only if the edges $\varphi(v_1)$ and $\varphi(v_2)$ share a vertex in $G$.  See an example in Figure \ref{exline}.

\begin{figure}[h]
\centering
\begin{subfigure}{0.45\textwidth}
\centering
\includegraphics[scale=0.3]{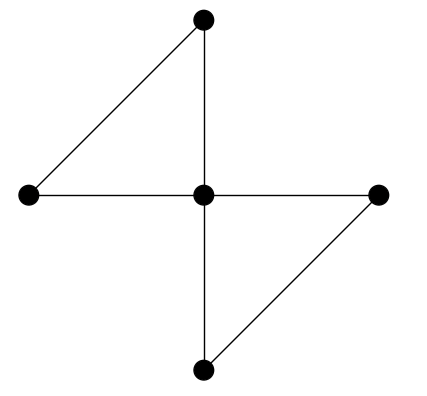}
\caption {Bowtie}
\end{subfigure}
\begin{subfigure}{0.45\textwidth}
\centering
\includegraphics[scale=0.3, angle=-45]{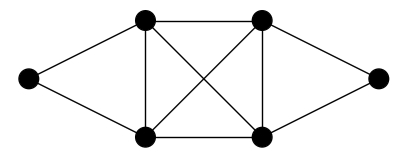}
\caption{Bowtie's Line Graph}
\end{subfigure}
\caption{Example of a Line Graph}
\label{exline} 
\end{figure}

Additionally, we will need the notion of an induced subgraph.  Given a graph $G=(V,E)$, an \textbf{induced subgraph}, $H$, on the vertex set $\tilde{V} \subset V$ has vertex set $\tilde{V}$ and edge set 
$$\tilde{E}= \left\{\{i,j\} | \{i, j\} \in \binom{\tilde{V}}{2} \cap E\right\}.$$ 
Throughout this paper, we will consider several graphs that lack certain induced subgraphs.  More precisely, a graph is said to be \textbf{$\boldsymbol{H}$-free} if it does not contain $H$ as an induced subgraph. Stanley's conjecture pertains to graphs that are claw-free, where a \textbf{claw} is defined as the complete bipartite graph $K_{1,3}$.

Let $P$ be a partially ordered set. We define the \textbf{incomparability graph} of $P$, also called the cocomparability graph in some texts, as the graph with vertex set $V=P$ and edge set $E$, where $\{x,y\} \in E$ if $x$ and $y$ are incomparable. 
A poset is said to be \textbf{$\boldsymbol{(a+b)}$-free}, if it does not contain an induced disjoint union of chains of length $a$ and $b$.
Note that the conditions claw-free and incomparability graph can also be expressed as being an incomparability graph of a ($3+1$)-free poset.
Another very important class of graphs are \textbf{natural unit interval graphs}, which are precisely incomparability graphs of ($2+2$) and ($3+1$)-free posets. Details on these graphs in general can be found in Sharesian and Wachs \cite{SharWa}.
For further definitions and information in the context of this paper, see Section \ref{genLol}.

Now consider a \textbf{partition}, a sequence of positive integers $\lambda = (\lambda_1, \lambda_2, \lambda_3, \ldots, \lambda_\ell)$ such that $\lambda_1 \geq \lambda_2 \geq \lambda_3 \geq \ldots \geq \lambda_\ell$. If $\sum_{i=1}^\ell \lambda_i = n $ we say $\lambda$ is a partition of $n$ and denote it $\lambda \vdash n$. The \textbf{conjugate} of $\lambda$ is defined as the partition $\lambda^\prime=(\lambda_1^\prime,\ldots,\lambda_{\lambda_1}^\prime)$ where $\lambda_i^\prime=|\{j:\lambda_j\geq i \}|$. 

Using partitions we can index symmetric functions and define some of the classical types of symmetric functions.
It should be noted that all of these functions form bases of the ring of symmetric functions, $\Lambda$, which we will consider later in the paper.
For more information about $\Lambda$  and symmetric functions in general, see \cite{MacDon} and \cite{Sag}.

Let $\lambda=(\lambda_1,...,\lambda_\ell)$ be a partition as before.  The \textbf{monomial symmetric function} corresponding to $\lambda$ is
\[
m_\lambda (x) = \sum_\alpha x_{1}^{\alpha_{1}} x_{2}^{\alpha_{2}} \ldots x_{n}^{\alpha_{\ell}}  
\]
where the sum is over all distinct permutations $\alpha = (\alpha_1, \alpha_2, \ldots, \alpha_\ell)$ of the entries of $\lambda= (\lambda_1, \lambda_2, \ldots, \lambda_\ell)$. 
Other very important symmetric functions, namely the \textbf{elementary} and \textbf{power sum symmetric functions}, are given by:
$$e_n (x)= m_{(1^n)}(x) = \sum_{i_1 < i_2 < \ldots < i_n} x_{i_1} x_{i_2} x_{i_3} \cdots x_{i_n}, $$ and
$$p_n (x) = m_{(n)}(x) = \sum_{i \geq 1} x_i^n $$ respectively. We will often omit the variables as understood. For $k<0$, we define $p_k = e_k =0$.
For a partition $\lambda \vdash n$ we define
$$e_\lambda = e_{\lambda_1}e_{\lambda_2}\cdots e_{\lambda_\ell}$$
and $p_\lambda$ is defined analogously. With the elementary symmetric functions the \textbf{Schur function} for a given partition $\lambda=(\lambda_1,\ldots,\lambda_\ell)$ can be obtained by the \textbf{dual Jacobi-Trudi identity} as follows:

$$ s_{\lambda}=\mathrm{det}(e_{\lambda^\prime_i-i+j}), $$
where $i,j \in [\ell]:=\{1,2,\ldots,\ell\}$. The notation $[\ell]$  will be used throughout the paper.

Note that given a partition $\lambda = (\lambda_1, \lambda_2, \lambda_3, \ldots, \lambda_\ell)$, we say a proper coloring is of type $\lambda$ if the number of vertices of each color written in descending order equals  $\lambda$.

A graph is said to be \textbf{$\boldsymbol{e}$-positive} if its corresponding chromatic symmetric function can be written as a linear combination of elementary symmetric functions with positive coefficients.

The importance of the power sum basis can be seen in the following formula.
Due to Stanley $\cite{Stan}$, we get the expansion 
\begin{equation}\label{eq: stanpow}
X_G = \sum_{S\subseteq E}(-1)^{|S|} p_{\lambda(S)},
\end{equation}where $\lambda(S)$ denotes the partition of $|V|$ whose parts correspond to the sizes of the connected components of the subgraph of $G$ with vertex set $V$ and edge set $S$.  If $G$ and $H$ are graphs, with $G\sqcup H$ denoting their disjoint union, it is known from Proposition 2.3 of Stanley \cite{Stan}:

\begin{equation}
    X_{G\sqcup H}=X_G \cdot X_H.
\end{equation}\label{eq: disjoint}
\par For section \ref{genLol}, we will also require the quasisymmetric refinement of the chromatic symmetric function, which was first stated by Shareshian and Wachs in \cite{SharWa}. Let $G=(V,E)$ be a graph. Then the \textbf{chromatic quasisymmetric function} is given by:
\begin{equation*}
    \widetilde{X}_G(t)= \sum_\kappa t^{\mathrm{asc}(\kappa)}x_{\kappa(v_1)}  x_{\kappa(v_2)} \cdots  x_{\kappa(v_n)},
\end{equation*}
where the sum is again over all proper colorings $\kappa$ and
\begin{equation*}
    \mathrm{asc}(\kappa)=\{\{v_i,v_j\} \in E : i<j \: \mathrm{and} \: \kappa(v_i)<\kappa(v_j)\}.
\end{equation*}
If $t=1$, we get the chromatic symmetric function. Consider again graphs $G,H$ and their disjoint union $G\sqcup H$. Then by \cite{SharWa}, following equation holds:
\begin{equation}
    \widetilde{X}_{G\sqcup H}(t) = \widetilde{X}_G(t) \cdot \widetilde{X}_H(t).
\end{equation}\label{eq: disjoint quasi}
For further definitions, information on which properties directly transfer from chromatic symmetric functions, and additional results on chromatic quasisymmetric functions, see Section \ref{genLol}, \cite{ChoHuh}, and \cite{SharWa}.

\section{\textit{e}-Positivity of Generalized Nets}\label{Nets}

Stanley \cite{Stan} conjectured that the claw-free incomparability conditions are sufficient for $e$-positivity.  However, he and Stembridge provided the net (Figure \ref{net}) as a counterexample to the necessity of these conditions \cite{StanStem}. We extend this counterexample to an infinite family of generalized nets, all of which are claw-free graphs, but none of which are $e$-positive.  More precisely, we define a generalized net as follows:
\begin{df}
A $\bf{generalized}$ $\bf{net}$ is a complete graph, $K_n$, $n \geq 3$,  with three additional vertices, where each vertex not in the $K_n$ is connected to a different vertex on $K_n$ via one edge.
We call the complete graph, $K_n$, the $\bf{body}$ of the $\bf{generalized}$ $\bf{net}$ and the additional vertices the $\bf{satellites}$.
\end{df}

An example of a generalized net can be seen in Figure \ref{fig: gennet}.

\begin{figure}
\centering
		\begin{subfigure}{0.45\textwidth}
			\centering
			\includegraphics[scale=0.4]{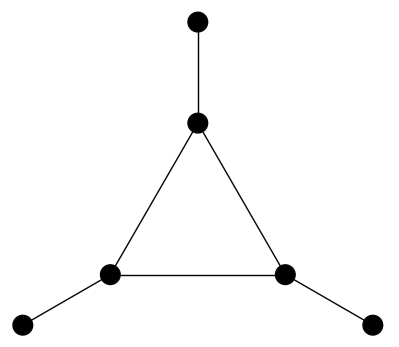}
			\caption {Net}\label{net}
		\end{subfigure}
		\begin{subfigure}{0.45\textwidth}
			\centering
			\includegraphics[scale=0.4]{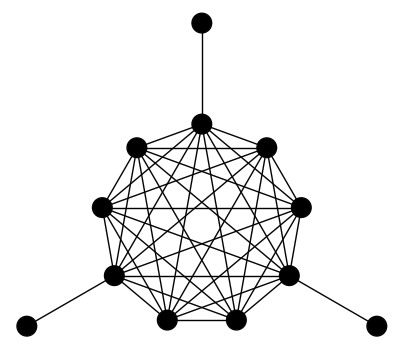}
			\caption{Generalized Net}\label{fig: gennet}
		\end{subfigure}
		\caption{}
\end{figure}

For simplicity in the following proofs, we say that a \textbf{connecting vertex} in a generalized net with body $K_n$, $n>2$ is a vertex with degree $n+1$.
This is in contrast to a \textbf{non-connecting} vertex in a generalized net, which we define to be a vertex with degree $n$.

\begin{theorem}
The chromatic symmetric function expanded in the elementary basis for a generalized net $G$ with a body of size $n,\; n\geq 3$ is:
$$
\begin{aligned}
        X_G =& (n+3)(n-1)!e_{(n+3)}+ 3(n^2-3)(n-2)!e_{(n+2,1)}\\
        &+ 6(n-1)(n-3)!e_{(n+1,2)} + 3(n^2-2n-1)(n-2)!e_{(n+1,1^2)} \\
        &+ 6(n-2)!e_{(n,2,1)}  - 6(n-3)!e_{(n,3)} +(n-3)(n-1)!e_{(n,1^3)}. 
\end{aligned}
$$
In particular, generalized nets are \emph{not} $e$-positive. 
\end{theorem}

\begin{proof}
	Given a generalized net, $W$, on $n+3$ vertices consider all possible proper colorings, case by case. The strategy behind this proof is straightforward: first, we determine all of the possible colorings of a graph, and then count the instances of each.  We consolidate this information into an expression for the chromatic symmetric function and then change the basis in order to derive a formula for the chromatic symmetric function in the $e$-basis.  Throughout this proof, we will use the term  ``trivial coloring'' to refer to a coloring in which every vertex has a different color.  
	Let a proper coloring of $W$ be of type $(\lambda_{1},\lambda_{2},\ldots,\lambda_{r},1^{k})$, where $\lambda_1\geq \lambda_2\geq \cdots \geq \lambda_r >1$ and
	\begin{equation}\label{eq: cond for r 1}
	    \sum_{i=1}^r{\lambda_{i}} + k = n+3.
	\end{equation}
	Notice that all vertices inside the body must have different colors, so at least $n$ colors are required. This gives the condition 
	\begin{equation}\label{eq: cond for r 2}
	    r + k \geq n .
	\end{equation}
	Next, since all $\lambda_{i} $ are strictly greater than 1, \begin{equation}\label{eq: cond for r 3}
	    \sum_{i=1}^{r}{\lambda_{i}} \geq {2r}.
	\end{equation}
	Combining the facts in (\ref{eq: cond for r 1}), (\ref{eq: cond for r 2}), and (\ref{eq: cond for r 3}) gives:
	\begin{align}
	   n - k &\geq {2r-3} \label{eq: cond for r 4}, \\
        r &\geq n-k. \label{eq: cond for r 5} 
	\end{align}
	It follows that
	\begin{equation*}
	    3 \geq r.
	\end{equation*}
	Additionally, using the statements in (\ref{eq: cond for r 1}), (\ref{eq: cond for r 4}), and (\ref{eq: cond for r 5}) one finds:
	\begin{equation*}
	    r+3 \geq \sum{\lambda_{i}} \geq {2r}.
	\end{equation*}

Considering each $r$ separately and using the restrictions listed, all possible types of proper colorings are given by: 
$$ (1^{n+3}), (2,1^{n+1}), (3,1^n), (4,1^{n-1}), (2,2,1^{n-1}),(3,2,1^{n-2}), (2,2,2,1^{n-3}).$$
This means that the chromatic symmetric function of generalized nets, written in terms of monomial symmetric functions, includes only terms corresponding to the types listed above. Furthermore, the coefficients in this linear combination count the possible colorings of the type given by the indexing partition. The coefficients can now be computed by counting all proper colorings. 
In order to make it easier to differentiate between the cases, write the calculations out in bullet points:

\begin{itemize}
	\item $\mathbf{(1^{n+3})}$: All $n+3$ vertices must be colored differently (Figure \ref{alldiff}). This gives $(n+3)!$ colorings.
   	\begin{figure}[H]
		\centering
		\includegraphics[scale=0.3]{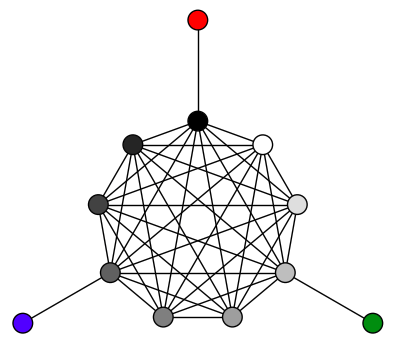}
		\caption{The trivial coloring}
		\label{alldiff} 
	\end{figure}
    \item $\mathbf{(2, 1^{n+1})}$: There are two possibilities for this type: either the vertices that share the first color are both satellites, or one is a satellite and the second one is in the body.  The first case (Figure \ref{case2a}) gives $3(n+1)!$ colorings and the second one (Figure \ref{case2b}) provides $3(n-1)(n+1)!$ colorings.  The total comes to $3n(n+1)!$.
    \begin{figure}[H]
		\centering
		\begin{subfigure}{0.45\textwidth}
			\centering
			\includegraphics[scale=0.3]{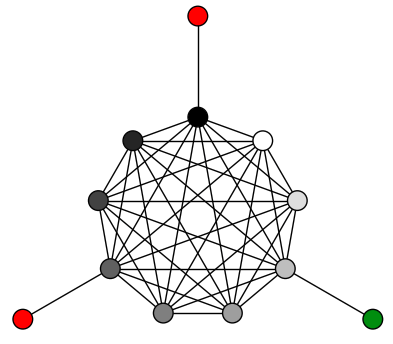}
			\caption {Repeated color appears on the satellites.}
            \label{case2a}
		\end{subfigure}
		\begin{subfigure}{0.45\textwidth}
			\centering
			\includegraphics[scale=0.3]{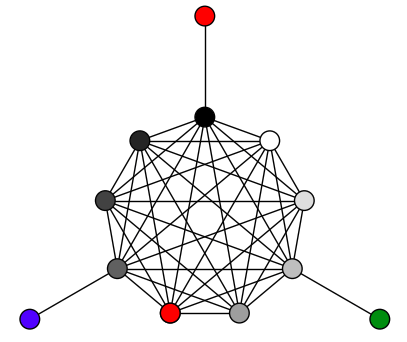}
			\caption{Repeated color appears once on a satellite and once in the body}
			\label{case2b}
		\end{subfigure}
		\caption{Two vertices share a color} 
	\end{figure}
    \item $\mathbf{(3,1^{n})}$: Two of the three vertices with the same color must be satellites, but there are two possible choices for positioning the third vertex: in the body, or on the remaining satellite. In the first case (Figure \ref{case3a}) there are three possible pairs of satellites and $n-2$ places for the vertex in the body. So in total, there are $3(n-2)n!$ options. In the second case (Figure \ref{case3b}), only the coloring of the body can vary which gives $n!$ possible colorings.  As a result, we have a total of $(3n-5)n!$ colorings.
    \begin{figure}[h]
		\centering
		\begin{subfigure}{0.45\textwidth}
			\centering
			\includegraphics[scale=0.3]{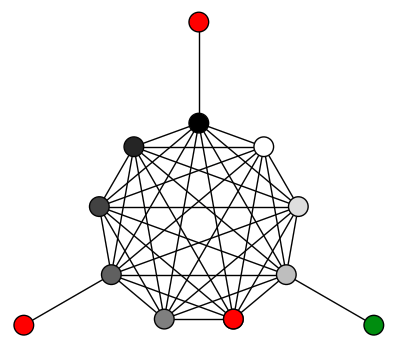}
			\caption {Two of the vertices that share a color are in satellites and one is in the body.}
            \label{case3a}
		\end{subfigure}
		\begin{subfigure}{0.45\textwidth}
			\centering
			\includegraphics[scale=0.3]{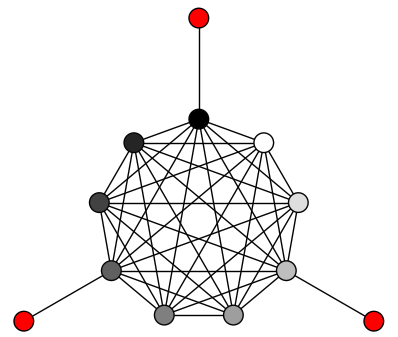}
			\caption{All three vertices that share a color are in satellites.}
			\label{case3b}
		\end{subfigure}
		\caption{Three vertices share a color} 
	\end{figure}
    \item $\mathbf{(4,1^{n-1})}$: The only way to achieve this type is to use the same color for all of the satellites and one non-connecting vertex in the body.  We color all the other vertices in distinct colors (Figure \ref{fig6}). There are $(n-3)$ places for the fourth vertex of the same color (the one in the body) and $(n-1)!$ possible ways to color the rest. This gives $(n-3)(n-1)!$ colorings.
    \begin{figure}[H]
		\centering
		\includegraphics[scale=0.3]{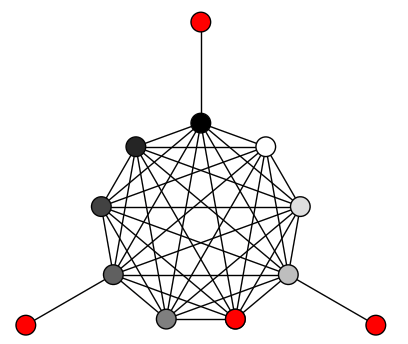}
		\caption{All satellites, and one vertex in the body, are colored the same color.}
		\label{fig6} 
	\end{figure}
    \item $\mathbf{(2,2,1^{n-1})}$: First possible case: Of the two repeated colors, the first color appears once in the body and once on a satellite while the other repeated color appears twice on satellites and once in the body (Figure \ref{fig: case5a}). There are two choices for the repeated color appearing on the satellites twice, three choices for other repeated color, $n-1$ choices for the fourth vertex in the body, and then $(n-1)!$ possibilities for the trivial colorings of the rest of the vertices. This gives a total of $6(n-1)(n-1)!$ colorings of this type.
    
   The second case occurs when both of the repeated colors appear once in the body and once on a satellite. There are three variants in which this can happen. The first sub-case, illustrated in Figure \ref{case5b}, occurs when two vertices connected to satellites are colored with one of the repeated colors. There are three choices of one color and a symmetric configuration for the other two colors. We color the remaining vertices distinctly. This gives a total of $6(n-1)!$ colorings. The second variant, illustrated in Figure \ref{case5c}, occurs when one vertex connected to a satellite is colored with one of the repeated colors.  There are 12 ways of arranging the satellites.  We choose the other two repeated colors arbitrarily in the body, with $(n-2)(n-1)!$ choices. This gives a total of $12(n-2)(n-1)!$ colorings. The last case, shown in Figure \ref{case5d}, occurs when none of the vertices connected to a satellite is colored with one of the repeated colors.  This case gives $6(n-2)(n-3)(n-1)!$, where the $6$ comes from different choices on the satellites, and the $(n-2)(n-3)(n-1)!$ comes from the different colorings of the body. Combining all of the above cases results in $6(n^2-2n+2)(n-1)!$ colorings.
    \begin{figure}[H]
		\centering
		\begin{subfigure}{0.45\textwidth}
			\centering
			\includegraphics[scale=0.3]{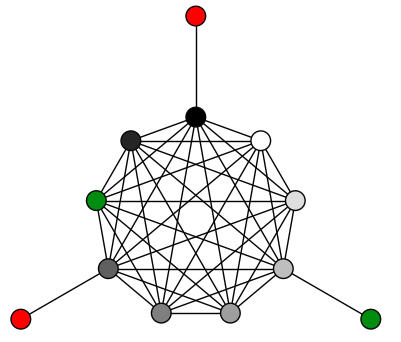}
			\caption {One repeated color appears twice in the body and the other repeated color appears once on a satellite and once in the body.}
            \label{fig: case5a}
		\end{subfigure}
		\begin{subfigure}{0.45\textwidth}
			\centering
			\includegraphics[scale=0.3]{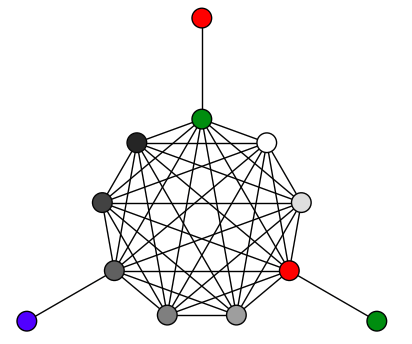}
			\caption{Both repeated colors appear once in the body and once on a satellite: case one.}\label{case5b}
		\end{subfigure}
		\centering
		\begin{subfigure}{0.45\textwidth}
			\centering
			\includegraphics[scale=0.3]{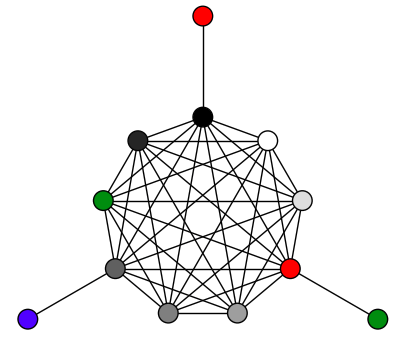}
			\caption {Both repeated colors appear once in the body and once on a satellite: case two}
            \label{case5c}
		\end{subfigure}
		\begin{subfigure}{0.45\textwidth}
			\centering
			\includegraphics[scale=0.3]{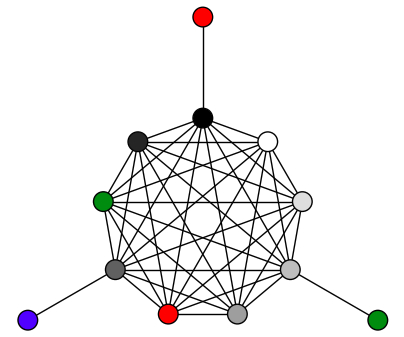}
			\caption{Both repeated colors appear once in the body and once on a satellite: case three}\label{case5d}
		\end{subfigure}
		\caption{Two pairs of vertices share a color.} 
	\end{figure}
    \item $\mathbf{(3,2,1^{n-2})}$: There is only one way to split the colors that appear between the vertices of the body and the satellites. However, there are two cases to consider: the vertex in the body of the first color can be on a connecting vertex, or not (Figure \ref{fig: case6}). 
    The first case gives $3(n-1)!$ colorings. The second one gives $3(n-3)(n-2)(n-2)!$. In total, $3(n^2-4n+5)(n-2)!$.
    \begin{figure}[H]
		\centering
		\begin{subfigure}{0.45\textwidth}
			\centering
			\includegraphics[scale=0.3]{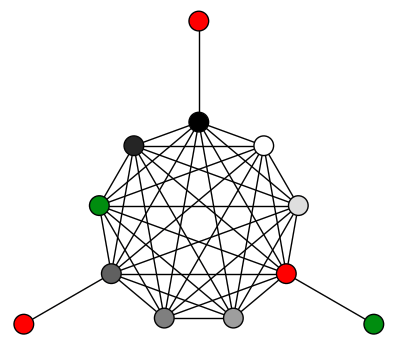}
			\caption {Repeated color appears on the satellites and once on a connecting vertex in the body.}
		\end{subfigure}
		\begin{subfigure}{0.45\textwidth}
			\centering
			\includegraphics[scale=0.3]{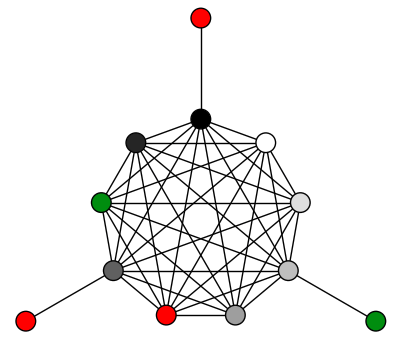}
			\caption{Repeated color appears once on a satellite and once on a non-connecting vertex in the body.}
		\end{subfigure}
		\caption{Three vertices share a color} \label{fig: case6}
	\end{figure}
    \item $\mathbf{(2,2,2,1^{n-3})}$: Split the case such that the number of connecting vertices which share a color with the satellites is $0,1,2$, or $3$.  This gives $4$ distinct possibilities.
    In the first sub-case (Figure \ref{fig: case7a}), there are 6 choices for how to color the satellites, as well as a choice of which non-connecting vertices will share a color with the satellites. There are $6(n-3)(n-4)(n-5)(n-3)!$ colorings of this type.
    
    In the second sub-case, we choose two pairs consisting of a satellite a connecting vertex attached to a different satellite, (Figure \ref{fig: case7b}). There are $3$ choices for the colors of the satellites and $2$ possible places for the non-repeated color. Adjusting the computation from the previous case accordingly gives $36(n-3)(n-4)(n-3)!$ colorings.
    
    In the third case (Figure \ref{fig: case7c}) there are $6$ ways to color the satellites, $3$ choices for the colors, and $3$ ways to place the pairs vertices sharing a color. This gives $54(n-3)(n-3)!$ colorings of this type.
    
    In the last case (Figure \ref{fig: case7d}), we choose $3$ colors for the satellites ($3!=6$ choices).  For each satellite coloring, there are two choices for the ways to color the connecting vertices so that they are colored differently from their respective satellites, and we color the remaining vertices distinctly ($(n-3!)$ choices). This leads to $12(n-3)!$ colorings.
    These facts lead to $6(n^3-6n^2+14n-13)(n-3)!$ colorings.
    \begin{figure}[H]
		\centering
		\begin{subfigure}{0.45\textwidth}
			\centering
			\includegraphics[scale=0.3]{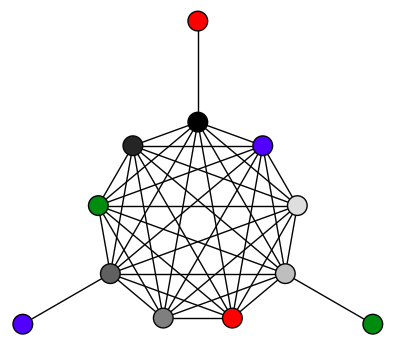}
			\caption {None of the vertices connected to satellites has the same color as any of the satellites.}\label{fig: case7a}
		\end{subfigure}
		\begin{subfigure}{0.45\textwidth}
			\centering
			\includegraphics[scale=0.3]{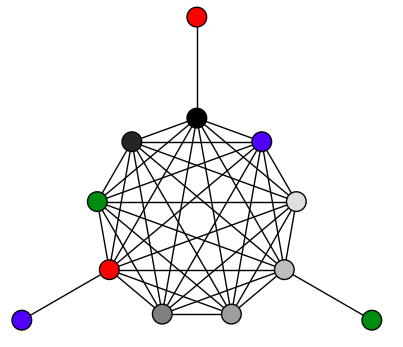}
			\caption{One of the vertices connected to a satellite has the same color as a satellite.}\label{fig: case7b}
		\end{subfigure}
		\centering
		\begin{subfigure}{0.45\textwidth}
			\centering
			\includegraphics[scale=0.3]{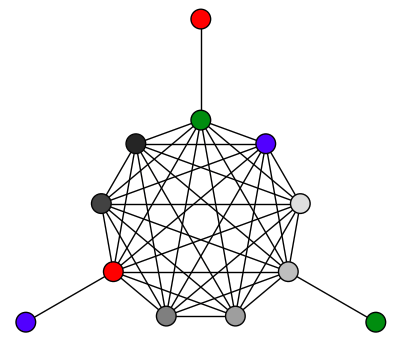}
			\caption {Two of the vertices connected to the satellites have the same colors as two of the satellites.}\label{fig: case7c}
		\end{subfigure}
		\begin{subfigure}{0.45\textwidth}
			\centering
			\includegraphics[scale=0.3]{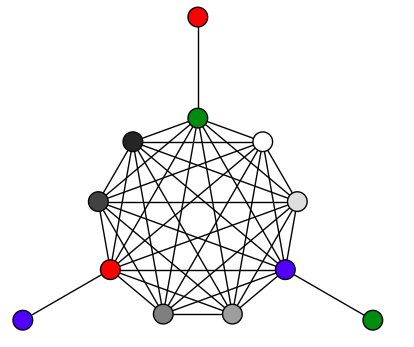}
			\caption{All three vertices connected to the satellites have the same colors as the satellites.}\label{fig: case7d}
		\end{subfigure}
		\caption{Three pairs of vertices share a color} 
	\end{figure}
\end{itemize}
	Combining these results gives: 
$$    
\begin{aligned}
    	X_G =& (n+3)!m_{(1^{n+3})}+ 3n(n+1)!m_{(2,1^{n+1})} 6(n^2-2n+2)(n-1)!m_{(2,2,1^{n-1})}  \\ &+ 6(n^3-6n^2+14n-13)(n-3)!m_{(2,2,2,1^{n-3})} + (3n-5)n!m_{(3,1^n)}  \\ &+ 3(n^2-4n+5)(n-2)!m_{(3,2,1^{n-2})} + (n-3)(n-1)!m_{(4,1^{n-1})} 
\end{aligned}
$$
        Now a change of basis from monomial to elementary symmetric functions is required in order to determine whether our chromatic symmetric function is $e$-positive. 
        By looking at the expression of monomial symmetric functions in terms of elementary symmetric functions one can verify using the well-known identity $e_\lambda= \sum M_{\lambda \mu}m_\mu$,  \cite{Sag}, that:
$$
\begin{aligned}
    	m_{(1^{n+3})} &= e_{(n+3)}\\
        m_{(2,1^{n+1})} &= e_{(n+2,1)} - (n+3)e_{(n+3)}\\
        m_{(2,2,1^{n-1})} &= e_{(n+1,2)}-(n+1)e_{(n+2,1)}+\tfrac{n(n+3)}{2}e_{(n+3)} \\ 
        m_{(2,2,2,1^{n-3})}  &= e_{(n,3)} - (n-1)e_{(n+1,2)}+ \tfrac{(n-2)(n+1)}{2}e_{(n+2,1)} - \\ 
        & -\tfrac{(n-2)(n-1)(n+3)}{6}e_{(n+3)}\\
        m_{(3,1^n)} &= e_{(n+1,1^2)} -2e_{(n+1,2)} - e_{(n+2,1)} + (n+3)e_{(n+3)}\\
        m_{(3,2,1^{n-2})} &= e_{(n,2,1)} - 3e_{(n,3)} - ne_{(n+1,1^2)} + 2(n-1)e_{(n+1,2)} +\\ 
        &+ (2n+1)e_{(n+2,1)} - (n+3)(n-1)e_{(n+3)}  \\ 
        m_{(4,1^{n-1})} &= e_{(n,1^3)}-3e_{(n,2,1)}+3e_{(n,3)}-e_{(n+1,1^2)}+\\
        &+2e_{(n+1,2)}+e_{(n+2,1)}-(n+3)e_{(n+3)}
\end{aligned}
$$
        Now using these expressions, the explicit formula for the chromatic symmetric function in the elementary basis for a generalized net with a body of size $n$ is:
$$        
    \begin{aligned}X_G =& (n+3)(n-1)!e_{(n+3)}+ 3(n^2-3)(n-2)!e_{(n+2,1)}\\
        &+ 6(n-1)(n-3)!e_{(n+1,2)} + 3(n^2-2n-1)(n-2)!e_{(n+1,1,1)} \\
        &+ 6(n-2)!e_{(n,2,1)}  - 6(n-3)!e_{(n,3)} +(n-3)(n-1)!e_{(n,1,1,1)}. 
    \end{aligned}
$$    

One can see that the $e_{(n,3)}$ term always has a negative coefficient, implying that generalized nets are not $e$-positive.
    
\end{proof}

\subsection*{Extension of (claw, $P_4$)-free Graphs}

Tsujie \cite{Tsuj} proved that (claw, $P_4$)-free graphs are $e$-positive. With our now established class of generalized nets we can show that Tsujie's $e$-positivity result does not extend to claw-free $P_4$-sparse graphs. $P_4$-sparse graphs were introduced by Ho\`{a}ng  \cite{Hoang} as a generalization of $P_4$-free graphs.  That is, every $P_4$-free graph is also $P_4$-sparse, but the converse does not hold. A graph $G$ is $P_4$-\textbf{sparse} if for every set of five vertices in $G$ there exists at most one induced $P_4$. These graphs are significant because they obey several structure theorems \cite{BM, GV, JO1, JO, Hoang}. As a further connection spiders and related graphs, note that Ho\`ang also proved that if $G$ is $P_4$-sparse, then $G$ is a (i) spider or co-spider, or (ii) $G$ has disconnected complement, or (iii) $G$ has a clique cutset \cite{Hoang}.  Therefore, a major subclass of $P_4$-sparse graphs consists of members related to spiders.

\begin{prop}
There are infinitely many claw-free, $P_4$-sparse graphs that are not $e$-positive.
\end{prop}

\begin{proof}
Generalized nets are claw-free and exhibit $P_4$-sparseness. If one selects five vertices in the body of the net, then the induced graph is $K_5$. On one hand, if one instead selects all three satellites, then the two remaining vertices are in the body. Hence, the induced subgraph necessarily contains one isolated vertex, and can contain at most one copy of $P_4$. Alternatively, if one selects three or more vertices that are part of the body, the induced subgraph either contains a copy of $K_i$ for $i=4$ or $5$, or it is isomorphic to a triangle with two additional edges attaching pendant vertices.  In the first case, there are no copies of $P_4$, and in the second, the induced subgraph contains exactly one copy of $P_4$. So the class of generalized nets gives an infinite family of claw-free, $P_4$-sparse graphs that are not $e$-positive.
\end{proof}

\section{Uniqueness of Generalized Spiders}\label{Uniq}
As mentioned before a natural question to ask is whether the chromatic symmetric function of a graph determines the graph up to isomorphism. We now give a class of graphs that are not trees, but whose members are still distinguished from each other by their chromatic symmetric functions. We will see that this class is an even further generalization of generalized nets, as we now allow multiple paths of arbitrary length instead of just three pendant vertices attached to the copy of $K_n$.
\begin{df}
A $\bf{generalized}$ $\bf{spider}$ is copy of $K_n, n \geq 3$, with paths of variable length attached such that each vertex in the body has at most degree $n$. These paths of length $(m_1,\ldots,m_\ell), \ell \leq n$ are called \textbf{legs}.
\end{df}

\noindent The name ``generalized spiders" was inspired by Martin's definition of a \textbf{spider} as a tree in which there is precisely one vertex $v$ with $\deg v\geq 3$ \cite{MarMorWag}.   Graphs can also be constructed by replacing the unique vertex, $v$ with $deg(v)\geq 3$, of a spider by a body of size $n \geq deg(v)$.
It is important to note that this is not the only connection to spiders as the following remark shows.

\begin{prop}\label{SpidLine}
G is a generalized spider with a body of size $n \geq 3$ if and only if G is the line graph of some spider.
\end{prop} 

\begin{figure}[ht]
\centering
\begin{subfigure}{0.45\textwidth}
\centering
\includegraphics[scale=0.3]{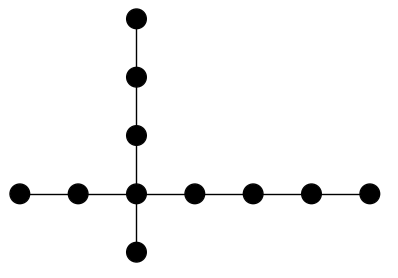}
\caption{Spider}
\end{subfigure}
\begin{subfigure}{0.45\textwidth}
\centering
\includegraphics[scale=0.32]{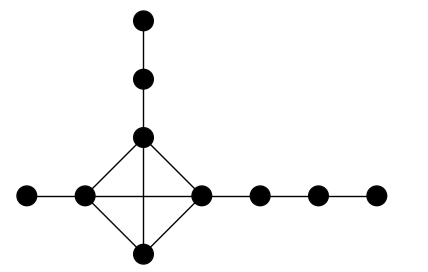}
\caption{Line Graph of Spider}
\end{subfigure}
\caption{}
\label{fig: spidergen} 
\end{figure}

\begin{proof}
Let $S=(V,E)$ be a spider and let $v \in V$ be the unique vertex with $deg(v) \geq 3$. Forming the line graph of $S$ produces a body of size $deg(v)$ as all adjacent edges share a vertex in $v$ and therefore are connected vertices in the line graph. It is easy to see that the line graph of a path of length $n$ is a path of length $n-1$. So if $S$ has legs of length $(m_1, \ldots, m_\ell)$ the corresponding generalized spider has legs of length $(m_1-1, \ldots, m_\ell-1)$. Legs of length 1 in $S$ are consequently just contributing to the size of the body. Reversing the line graph construction (with the convention that for $n=3$ the $K_3$ must be transformed to a claw not a $K_3$) shows that a generalized spider can be transformed to a spider. An illustration can be seen in Figure \ref{fig: spidergen}.
\end{proof}

With this observation we can now use the uniqueness result regarding spiders to prove uniqueness of the chromatic symmetric function of their generalized counterparts. 

\begin{theorem}
The chromatic symmetric function uniquely distinguishes spiders \cite{MarMorWag}.
\end{theorem}

We know that we can extract information about the graph from the chromatic symmetric function. More specifically, the number of $k$-matchings and the number of independent sets can be found. A $k$-\textbf{matching} is a set of $k$ independent edges, i.e. no two vertices in the set share a vertex. An \textbf{independent (or stable) set} is defined similarly as a subset of vertices $I$ in a graph such that no two vertices in $I$ are adjacent. These sets have associated polynomials.

\begin{df}
If $m_k$ is the number of $k$-matchings in a graph, then the \textbf{matching polynomial} is $$\mu(x)=\sum_{k\geq0}^{|E|}m_kx^k.$$

\end{df}

The value $m_k$ can be recovered from the chromatic symmetric function expanded in the power sum basis stated in equation (\ref{eq: stanpow}).  Because knowledge of $m_k$ for each $k$ determines $\mu(x)$, the matching polynomial can be recovered from the chromatic symmetric function \cite{OrScott}.
We define the independence polynomial of $G$ as:
\begin{equation*}
I(x)=\sum_{i=0}^{|V|}\Phi_ix^i,
\end{equation*}
where $\Phi_i$ stands for the number of independent sets of size $i$. Note that $I(x)$ is sometimes referred to as the stable set polynomial, see \cite{Stan1}.\\ 

Though it is a well-known fact that the independence polynomial is determined by the chromatic symmetric function, we have been unable to find a proof in the literature.  Thus, we provide one for completeness.

\begin{lemma}\label{IndCSF}
For a graph $G$ the number of independent sets of a certain size, $k$, can be recovered from its corresponding chromatic symmetric function.
\end{lemma}
\begin{proof}
Let $G=(V,E)$ with $|V|=n$ be a graph. If $k$ vertices in a proper coloring of $G$ have the same color, then they necessarily are an independent set.  Since the coefficients of the chromatic symmetric function expanded in the monomial basis directly correspond to colorings, the coefficient of the term $m_{(k,1^{n-k})}$ determines the number of independent sets of size $k$ up to a permutation of the colors. So by dividing the coefficient by $(n-k)!$ we get the number of independent sets of size $k$.
\end{proof}

Finally we need some more graph-theoretic results in order to conclude that the chromatic symmetric functions of generalized spiders are unique.
As mentioned in \cite{CS}, the characteristic polynomial of the adjacency matrix of a tree is equal to its matching polynomial. Therefore the result in \cite{LepGut} that non-isomorphic spiders have unique characteristic polynomials implies that they have distinct matching polynomials. Furthermore the matching polynomial of a graph is the independence polynomial of its line graph. This can easily be seen by noting that all edges in a $k$-matching on a graph directly correspond to an independent set of its line graph.  We are now ready to state the main theorem of the section:

\begin{theorem}
No two generalized spiders have the same chromatic symmetric function.
\end{theorem}
\begin{proof}
Let $G$ be a generalized spider. The independence polynomial $I(x)$ of $G$ is unique among generalized spiders, because by Proposition \ref{SpidLine}, $G$ is the line graph of a spider, and spiders have unique matching polynomials. Therefore by Lemma \ref{IndCSF}, $X_G$ is uniquely determined. 
\end{proof}

\section{Horseshoe Crab Graphs}\label{genLol}
As Guay-Paquet showed in \cite{GuayPaq} it is sufficient to show $e$-positivity for the incomparability graphs of $(2+2)$ and $(3+1)$-free posets in order to resolve Stanley's $e$-positivity conjecture. As mentioned in Section \ref{back} these graphs correspond exactly to the natural unit interval graphs (defined below). In the following section we consider the $e$-positivity question for another class of natural unit interval graphs. These graphs we will call horseshoe crab graphs, an allusion to the fact that their appearance resembles that of horseshoe crabs. An illustration can be seen in Figure \ref{fig: Incomp} and a definition is in Definition \ref{hcg}. Note that in biology horseshoe crabs, like spiders, are phylum {\em arthropoda} thus are indeed spider kin. We will use our consideration of horseshoe crabs to explore an even larger class of $e$-positive natural unit interval graphs, which leads to the main conjecture (see Conjecture \ref{theo: m1_m2_m3,n,...,n}) of this section. The technique we use to consider the $e$-positivity of our horseshoe crab graphs was first introduced by Cho and Huh in \cite{ChoHuh}. In our work with these graphs we also generalize their method of weight preserving injections in Lemma \ref{lem: weight shift}.

In order to state and prove our result, we need to provide some further background. First we need to introduce natural unit interval orders. Further information on natural unit interval orders and their properties can be found in Sharesian and Wachs \cite{SharWa}.  A natural unit interval order is a partial ordering $P(\mathbf{m})$ on the elements of $[n]$, which is induced by a sequence $\mathbf{m}=(m_1,\ldots,m_{n-1})$ of positive non-decreasing integers. The sequence must satisfy $i \leq m_i \leq n$ for all $i \in [n-1]$. Then the corresponding order relation $<_P$ is given by
\begin{equation*}
   i <_P j \: \textrm{ if } \: m_i < j \leq n.
\end{equation*}
\begin{df}
The incomparability graph of $P(\mathbf{m})$ is called a $\bf{natural}$ $\bf{unit}$ $\bf{interval}$ $\bf{graph}.$
\end{df}

Using the sequence $\mathbf{m}$ it is easy to characterize the edges of the incomparability graph since, for $i<j$:
\begin{align*}
    i \: \mathrm{and} \: j \: \textrm{are incomparable if and only if} \: m_i \geq j.
\end{align*}
In other words $i$ is incomparable to every $j$ such that $i < j\le m_i$ and each $k$ such that $k<i\le m_k$. 

\begin{df}
Given $\mathbf{m}=(2,m_2,m_3,n,\ldots,n)$, the incomparability graph of $P(\mathbf{m})$ is called a $\bf{ horseshoe}$ $\bf{crab}$ $\bf{graph}$.
\label{hcg}
\end{df}
\noindent
An example of the horseshoe crab graphs can be seen in Figure \ref{fig: Incomp}. Recall now the definition of the chromatic quasisymmetric function as given in Section \ref{back}. For natural unit interval graphs, the chromatic quasisymmetric function has the following properties. This is Theorem 4.5 and Corollary 4.6 of \cite{SharWa}.

\begin{theorem} \label{theo: palindrom}
\cite{SharWa} Let $G=(V,E)$ be a natural unit interval graph. Then
$$ \widetilde{X}_G(t) \in \Lambda[t]$$
so the coefficients of $t^i$ are symmetric functions. Furthermore the coefficients form a palindromic sequences, i.e. $\widetilde{X}_G(t)=t^{|E|}\widetilde{X}_G(t^{-1})$. 
\end{theorem}
We can now state a necessary and sufficient condition for $\widetilde{X}_G(t)$ being $e$-positive. Let $G$ be a natural unit interval graph and 
$$\widetilde{X}_G(t)=\sum_{i=1}^m f_it^i = \sum_{\lambda \vdash n} E_\lambda(t)e_\lambda, \qquad f_i \in \Lambda.$$
Then $\widetilde{X}_G(t)$ is $e$-positive if and only if the polynomial $E_\lambda(t)$ has non-negative coefficients for all $\lambda$ (see \cite{ChoHuh}).

Schur-positivity for natural unit interval graphs has been proven for $t=1$ by Gasharov in \cite{Gash}.
Shareshian and Wachs extended this result for the quasisymmetric refinement in \cite{SharWa}. In order to state this result and also use it for proving $e$-positivity, we need to define $P$-tableaux, which were used by Gasharov in \cite{Gash} in order to prove Schur-positivity of (3+1)-free posets.

\begin{df}
Let $P$ be a poset on $n$ elements and $\lambda \vdash n$. A $\boldsymbol{P}\bf{-tableau}$ $\bf{of}$ $\bf{shape}$ $\boldsymbol{\lambda}$ is a filling of a Young diagram of shape $\lambda$ (in English notation, i.e. each row is placed below the previous one) with elements of $P$ such that:
\begin{itemize}
    \item Each element of $P$ appears exactly once
    \item The rows are strictly increasing ( $a <_P b$ if $a \in P$ appears to the left of $b \in P$)
    \item The columns are pairwise non-decreasing ($b \not<_P a$ if $a \in P$ appears immediately above $b \in P$)
\end{itemize}
Let $G=(V,E)$ be the incomparability graph of $P$ and $T$ be a $P$-tableau. An edge $\{i,j\}$ with $i<j$ and $i$ appearing above $j$ in $T$ is called a $\boldsymbol{G}\bf{-inversion}$. The number of $G$-inversions is denoted by $\mathrm{inv}_G(T)$ and referred to as the $\bf{weight}$ $\bf{of}$ $\bf{the}$  $\boldsymbol{P}\bf{-tableaux}$. (Definition 6.1 in \cite{SharWa})
\end{df}
\noindent
Equivalently $\mathrm{inv}_G(T)$ can be described as the number of incomparable pairs $(a,b)$ with $a<b$ and $b$ appearing above $a$ in $T$. For example for the sequence $\mathbf{m}=(2,4,6,8,8,8,8)$ a possible $P$-tableau of shape $(3,2,1^3)$ and with weight 5 is shown in Figure \ref{fig:p-tab}. The weight has been calculated as follows:
$$\mathrm{inv}_G=|\{2,3\},\{\{4,5\},\{4,6\},\{4,7\},\{5,7\},\{6,7\}\}|=6.$$
The corresponding natural unit interval graph can be seen in \ref{fig: Incomp}. As it will be useful for a later proof, a general filling of a $P$-tableau is given in Figure \ref{fig: general tab}.
\begin{figure}[h]
\captionsetup[subfigure]{justification=centering}
    \centering
    \begin{subfigure}{0.15\textwidth}
    \centering
    \ytableausetup{boxsize=2em}
    \begin{ytableau}
    1&3&7\\
    2&5\\
    6\\
    4\\
    8
    \end{ytableau}
    \subcaption{A $P$-tableau}
    \label{fig:p-tab}
    \end{subfigure}
    \begin{subfigure}{0.5\textwidth}
    \centering
    \includegraphics[scale=0.5]{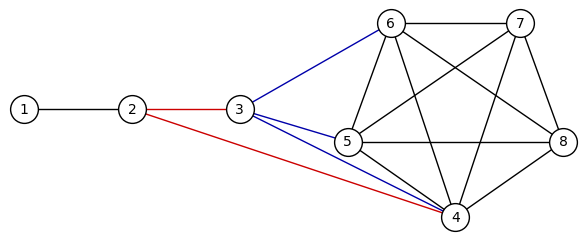}
    \subcaption{The incomparability graph of $P$;\\a horseshoe crab graph} \label{fig: Incomp}
    \end{subfigure}
    \begin{subfigure}{0.2\textwidth}
    \centering
    \ytableausetup{mathmode, boxsize=2.5em}
    \begin{ytableau}
    a_{1,1} & a_{1,2} & a_{1,3} & \dots \\
    a_{2,1} & a_{2,2} & \dots \\
    a_{3,2} & \vdots \\
    \vdots
    \end{ytableau}
    \subcaption{A general $P$-tableau}\label{fig: general tab}
    \end{subfigure}
    \caption{Sequence $\mathbf{m}=(2,4,6,8,8,8,8)$}
\end{figure}

With this definitions we can now state a very important result which is Theorem 6.3 in \cite{SharWa}.
\begin{theorem}\label{theo: s-pos}
\cite{SharWa} For a natural unit interval order $P$ the chromatic quasisymmetric function of its incomparability graph $G$ has the following expansion:
$$\widetilde{X}_G(t)=\sum_T t^{\mathrm{inv}_G(T)}s_{\lambda(T)},$$
where the sum is over all $P$-tableaux and $\lambda(T)$ denotes the shape of $T$. Hence $\widetilde{X}_G(t)$ is Schur-positive. 
\end{theorem}

From this formula, the corresponding expansions in the elementary basis can be found using the dual Jacobi-Trudi identity. Furthermore Shareshian and Wachs proved positivity of a certain coefficient by giving an explicit formula, which leads to the following theorem which is Corollary 7.2 in \cite{SharWa}.

\begin{theorem}\label{theo: E_n pos}
\cite{SharWa} Let $G$ be a natural unit interval graph on $n$ elements and $\widetilde{X}_G(t)=\sum_{\lambda \vdash n}E_\lambda(t)e_\lambda$. Then $E_{(n)}(t)$ is a positive polynomial in $t$. 
\end{theorem}

Having established all these results, we can now use the technique of weight preserving injections to consider $e$-positivity of a class of natural unit interval orders corresponding to a certain sequence.
This leads on from the work of Cho and Huh \cite{ChoHuh} who established a number of significant results including the following Remark \ref{rem: known}.  Like Cho and Huh, our work centres around a body with an appendage (the ``tail'' of the horseshoe crab in our case) and it would be interesting to further explore which combinations of bodies and appendages are $e$-positive.

\begin{remark}\label{rem: known}
\begin{itemize}
    \item Let $\mathbf{m}=(m_1,m_2,n,\ldots,n)$. For the incomparability graph $G$ of the natural unit interval order $P(\mathbf{m})$, $\widetilde{X}_G(t)$ is $e$-positive. Therefore $X_G$ is $e$-positive. \cite{SharWa}
    \item Let $\mathbf{m}=(r,m_2,m_3,\ldots,m_r,n,\ldots,n)$ and $G$ be the incomparability graph of $P(\mathbf{m})$. Then $\widetilde{X}_G(t)$ and thus $X_G$ is $e$-positive. \cite{ChoHuh}
\end{itemize}
\end{remark}

\begin{conj}\label{theo: 2,m2,m3}
Let $\mathbf{m}=(2,m_2,m_3,n,\ldots,n)$  and $G$ be the incomparability graph of $P(\mathbf{m})$. Then $\widetilde{X}_G(t)$ and thus $X_G$ is $e$-positive.
\end{conj}

\begin{proof}[Partial proof.]
\renewcommand\qedsymbol{}
We show all but one of the coefficients are positive. First note we may assume this graph is connected, as we have a clique of size $(n-3)$, so as the chromatic quasi-symmetric function (CQSF) is the product the CQSFs of its connected components any disconnected graph has a CQSF written as the product of the CQSF of a complete graph and the CQSF of a 3 vertex claw-free incomparability graph. It follows from Remark \ref{rem: known} that the former is $e$-positive and all cases in the latter are easily computed. We are then done as the product of $e$-positive functions is clearly $e$-positive. So we then have that the graph is connected and hence $2\not<_P 3$, $3\not<_P4$. Let $G$ denote the incomparability graph of $P(\mathbf{m})$ on $n$ elements with $\mathbf{m}=(2,m_2,m_3,n,\ldots,n)$. As the only non-maximal elements of $P(\mathbf{m})$ are 1, 2, and 3, there are at most 3 rows of length greater than 1. Since a maximal chain in $\nat$ has at most 3 elements, the first row has length at most 3. Since any 3-chain starts with 1 and 3, the presence of the 3-chain leaves room for at most one other 2-chain, leaving the possible shapes as:
$$(1^n),(2,1^{n-2}),(2^2,1^{n-4}),(2^3,1^{n-6}),(3,1^{n-3}),(3,2,1^{n-5}).$$
This leads to following expansion of the chromatic quasisymmetric function by Theorem \ref{theo: s-pos}, where we use $S$ to denote the appropriate coefficient:
\begin{align*}
\widetilde{X}_G(t) &= S_{(1^n)}(t)s_{(1^n)}+S_{(2,1^{n-2})}(t)s_{(2,1^{n-2})}+S_{(2^2,1^{n-4})}(t)s_{(2^2,1^{n-4})}\\
&+S_{(2^3,1^{n-6})}(t)s_{(2^3,1^{n-6})}+S_{(3,1^{n-3})}(t)s_{(3,1^{n-3})}+S_{(3,2,1^{n-5})}(t)s_{(3,2,1^{n-5})}.
\end{align*}
By applying dual Jacobi-Trudi identity, the coefficients for elementary symmetric functions can be calculated:
\begin{align}
    \setcounter{equation}{0}
    E_{(n)}(t) &= S_{(1^n)}(t)+S_{(3,1^{n-3})}(t)-S_{(2,1^{n-2})}(t)\label{coef: cov theo}\\
    E_{(n-1,1)}(t) &= S_{(2,1^{n-2})}(t)+S_{(3,2,1^{n-5})}(t)-S_{(2^2,1^{n-4})}(t)-S_{(3,1^{n-3})}(t)\label{coef: 2 in 2}\\
    E_{(n-2,2)}(t) &= S_{(2^2,1^{n-4})}(t)-S_{(2^3,1^{n-6})}(t)-S_{(3,1^{n-3})}(t)\label{coef: 2 in 1}\\
    E_{(n-2,1^2)}(t) &= S_{(3,1^{n-3})}(t)-S_{(3,2,1^{n-5})}(t)\label{coef: 1 in 1 easy}\\
    E_{(n-3,3)}(t) &= S_{(2^3,1^{n-6})}(t) -S_{(3,2,1^{n-5})}(t)\label{coef: 1 in 1 hard}\\
 E_{(n-3,2,1)}(t) &= S_{(3,2,1^{n-5})}(t) 
    \label{coef: triv}
\end{align}
In order to prove the $e$-positivity of $G$ it is sufficient to show that the coefficients of all the polynomials
$E_\lambda(t)$ are positive. We will show all coefficients expect $E_{(n-1,1)}$ are positive. By Theorem \ref{theo: E_n pos}, $E_{(n)}(t)$ is a positive polynomial. Furthermore since by Theorem \ref{theo: s-pos}
\begin{equation}\label{coef: s}
S_\lambda(t)=s_\lambda \cdot \sum_{T \in \mathcal{T}_{\lambda}}t^{\mathrm{inv}_G(T)}, 
\end{equation}
where $\mathcal{T}_{\lambda}$ denotes the set of $P$-tableaux of shape $\lambda$, the coefficient $E_{(n-3,2,1)}(t)$ in (\ref{coef: triv}) is trivially positive. For the remaining coefficients, positivity can be shown by the use of weight preserving injections. Those will be defined for the individual cases and the notation of the general $P$-tableau in Figure \ref{fig: general tab} will be used throughout:
\begin{itemize}
   \item  $\mathbf{E_{(n-2,2)}(t)}$: In order to define a weight preserving injection
    $$ \xi: \mathcal{T}_{(2^3,1^{n-6})} \cup \mathcal{T}_{(3,1^{n-3})} \to \mathcal{T}_{(2^2,1^{n-4})} $$
    a few cases need to be considered. For $T \in \mathcal{T}_{(2^3,1^{n-6})}$ a straightforward mapping rule is moving the $a_{3,2}$ element exactly above the $a_{4,1}$ element. Since 1,2 and 3 are the only possible ways to start a 2-chain in $\nat$, this does not violate being pairwise non-decreasing along the columns. Furthermore it is trivially weight preserving as the relative positions of the incomparable elements do not change. Note that in the image of $\mathcal{T}_{(2^3,1^{n-6})}$ under $\xi$,  $(a_{1,1},a_{2,1},a_{3,1})$ forms a certain permutation of $(1,2,3)$. Now consider $P$-tableaux of shape $(3,1^{n-3})$. Note that we always have $a_{1,1}=1$ and $a_{1,2}=3$ as all 3-chains start like this. The tableaux of this shape are mapped as follows:
    \begin{itemize}

        \item $a_{2,1} \neq 2, a_{3,1} \neq 2$, and $3 \not<_P a_{2,1}$: By moving the $a_{2,1}$ up next to the 1 and dropping $3,a_{1,3}$ to form the new second row a $P$-tableau of shape $(2^2,1^{n-4})$ is obtained. Note that this transformation fixes the first element of the third row, which was stipulated not to be 2. It is also not 1 or 3 as these were sent to the second row. We then conclude it is not in the image of $\mathcal{T}_{(2^3,1^{n-6})}$ as we saw the first element of the first 3 rows are 1, 2, and 3 (not necessarily in that order). We gain an inversion from $\{3,a_{2,1}\}$ and since  $3 \not<_P a_{2,1}$ and $3<_P a_{1,3} \implies a_{2,1}<a_{1,3}$ we lose one from $\{a_{2,1},a_{1,3}\}$ .
        \item $a_{2,1} \neq 2, a_{3,1} \neq 2$, and $3 <_P a_{2,1}$: This case can be solved by simply moving the 3 in front of the $a_{2,1}$ element. This does not add or subtract any $G$-inversions because $3<_P a_{2,1}$. This can not appear in the image of $\mathcal{T}_{(2^3,1^{n-6})}$ as $a_{2,1} \neq 2$ and $a_{3,1} \neq 2$. This does not overlap with the previous case as the second element in the first row of the image in the previous case was stipulated to be incomparable with 3, and here it is $<_P3$.
        \item $a_{3,1}=2$: We move 2 to the start of the second row, move $a_{1,3}$ in front of the 2, and drop $a_{2,1}$ to form the 3rd row. These do not overlap with the other cases as this is the only case where the second element of the first row is a 3. We lose the $\{3,a_{2,1}\}$ inversion (these are in fact incomparable as $a_{2,1}$ appears below 2 so it is incomparable with 2 and certainly incomparable with 3) and gain the $\{a_{1,3},a_{2,1}\}$ inversion ($a_{1,3}>_P 3$ since it appears to the right of it, and we just established $2\not<_P a_{2,1}$ so $a_{1,3}>a_{2,2}$).
        \item $a_{2,1}=2$: We map 2 to be the first entry of the first row, $a_{1,3}$ to be the second entry of the first row, and 1 and 3 to be the first and second entries of the second row respectively. We leave all other entries in their original place. We gain a $\{1,2\}$ inversion and lose the $\{2,3\}$ inversion. $a_{1,3}$ does not affect the weight as $3<_P a_{1,3}$.
       \end{itemize}

    \begin{figure}
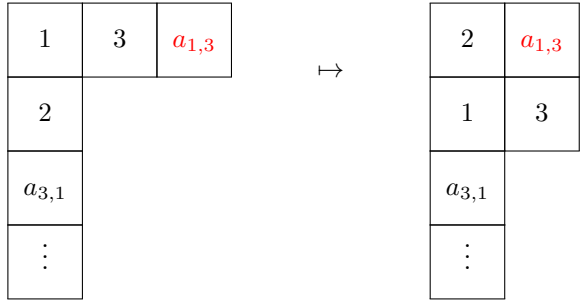

        \centering
        \begin{tabular}{c c c}
            \ytableausetup{mathmode, boxsize=2.75em}
             \begin{ytableau}
                1 & 3 & \color{red}a_{1,3}\\
                2 \\
                a_{3,1}\\
                \vdots
            \end{ytableau}&
            \qquad $\mapsto$ &
            \qquad
            \begin{ytableau}
            2  & \color{red}a_{1,3}\\
            1 & 3 \\
            a_{3,1}\\
            \vdots
        \end{ytableau}
        \end{tabular}
        \caption{Part of the map $\xi$}
        \label{fig: 2 in 1 last case}
    \end{figure}
    
    \item $\mathbf{E_{(n-2,1^2)}(t)}$: Recalling equation (\ref{coef: 1 in 1 easy}), $\mathcal{T}_{(3,2,1^{n-5})}$ needs to be injected into $\mathcal{T}_{(3,1^{n-3})}$. As mentioned before the $P$-tableaux of shape $(3,2,1^{n-5})$ have a lot of structure due to the fact that 1,2 and 3 are the only non-maximal elements of $\nat$. Therefore a very straightforward weight-preserving injection $\eta$ is illustrated in Figure \ref{fig: 1 in 1 easy}. Note that the image is a $P$-tableau as $a_{2,2},a_{3,1}\ge 4$ and hence $a_{3,1} \not<_P a_{2,2}$. Furthermore as the relative positions of the elements did not change, the number of inversions is the same for both $P$-tableaux. For injectivity consider $T_1,T_2 \in \mathcal{T}_{(3,2,1^{n-5})}$ with $\eta(T_1)=\eta(T_2)$. Since the original $P$-tableau can be recovered by moving the element directly underneath the 2 up, it follows that $T_1=T_2$. Thus $\eta$ is injective.  
    \begin{figure}[ht]
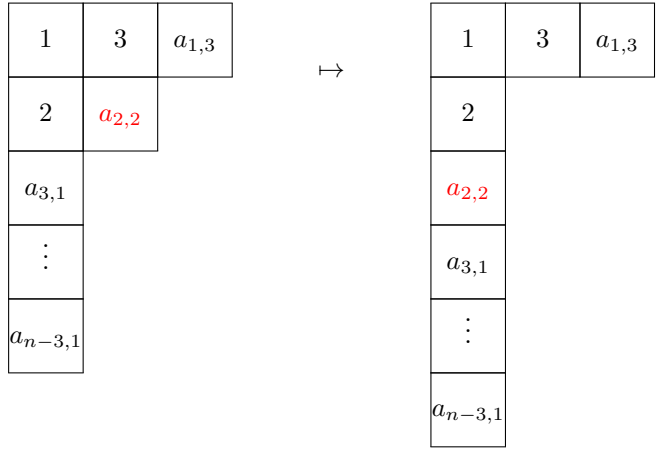

        \centering
        \begin{tabular}{c c c}
            \ytableausetup{mathmode, boxsize=2.75em}
             \begin{ytableau}
                1 & 3 & a_{1,3}\\
                2 & \color{red}a_{2,2}\\
                a_{3,1}\\
                \vdots\\
                a_{n-3,1}
            \end{ytableau}&
            \qquad $\mapsto$ &
            \qquad
            \begin{ytableau}
            1 & 3 &a_{1,3}\\
            2 \\
           \color{red}a_{2,2}\\
            a_{3,1}\\
            \vdots\\
            a_{n-3,1}
        \end{ytableau}
             
        \end{tabular}
        \caption{The map $\eta$}
        \label{fig: 1 in 1 easy}
    \end{figure}
\end{itemize}
\end{proof}

For the last coefficient $E_{(n-3,3)}(t)$ we need to establish some additional results in order to use the method of injections. First of all we define Schur-unimodality for quasisymmetric functions.

\begin{df}
Let $G$ be a graph and $\widetilde{X}_G(t)=\sum_{i=0}^{m} f_i t^i$ for $f_i \in \Lambda$. $\widetilde{X}_G(t)$ is said to be $\bf{Schur-unimodal}$ if $f_{i+1}-f_i$ is Schur-positive whenever $0 \leq i \leq \frac{m-1}{2}$.
\end{df}

The next theorem gives a powerful result involving Schur-unimodality. The result that leads to this was first conjectured by Shareshian and Wachs in \cite{SharWa}, then proved by Brosnan and Chow in \cite{BroCho} and later by Guay-Paquet \cite{Guay2Paq} using a different method.

\begin{theorem}
\cite{BroCho} \cite{Guay2Paq} Let $G$ be a natural unit interval graph. Then $\widetilde{X}_G(t)$ is Schur-unimodal.
\end{theorem}\label{theo: s-unimod}

With this theorem and the palindromicity result stated in Theorem \ref{theo: palindrom} we can now generalize the method of weight preserving injections.

\begin{lemma}\label{lem: weight shift}
If there exists an injection $\psi : \mathcal{T}_{\lambda_1} \to \mathcal{T}_{\lambda_2}$ and a constant $c\in \mathbb{Z}$ such that:
$$\forall T \in \mathcal{T}_{\lambda_1}: \mathrm{inv}_G(\psi(T))=\mathrm{inv}_G(T) + c .$$
Then the coeffiecients of $S_{\lambda_2}$ dominate the coefficients of $S_{\lambda_1}$ or equivalently there are at least as many $P$-tableaux of shape $\lambda_2$ as there are of shape $\lambda_1$ for any particular weight.
\end{lemma}\label{lem: constant weight shift}

\begin{proof}
Let $S^{(j)}_{\lambda_i}$ be the coefficient for the term $t^js_{\lambda_i}$. To show that $S^{(j)}_{\lambda_1}\le S^{(j)}_{\lambda_2}$, notice that the weight shifting injection shows $S^{(j)}_{\lambda_1}\le S^{(j+c)}_{\lambda_2}$. First consider the case that $c\ge 0$. By palindromicity the terms are symmetrical so it is enough to prove it for the second half of the terms, that is to show that when  $j>\frac{|E|}{2}$ we have
$$S^{(j)}_{\lambda_1}\le S^{(j)}_{\lambda_2}.$$ 
Since
$$S^{(j)}_{\lambda_1}\le S^{(j+c)}_{\lambda_2}$$
and $j>\frac{|E|}{2}$, palindromicity and Schur-unimodality result in
$$S^{(j+c)}_{\lambda_2}\ge S^{(j)}_{\lambda_2},$$
as desired.\\
If $c$ is negative take $j\le \frac{|E|}{2}$ and get
$$S^{(j)}_{\lambda_1}\le S^{(j+c)}_{\lambda_2}\le S^{(j)}_{\lambda_2}$$
by Schur-unimodality.
\end{proof}
.
Using Lemma \ref{lem: constant weight shift} we can now continue the analysis of Conjecture \ref{theo: 2,m2,m3}.

\renewcommand*{\proofname}{Continuation of the partial proof of Conjecture \ref{theo: 2,m2,m3}}
\begin{proof}
The last neccessary step in order to prove $e$-positivity of $\widetilde{X}_G(t)$ is showing that the polynomial $E_{(n-3,3)}(t)$ only has positive coefficients. By Lemma \ref{lem: constant weight shift} and equation (\ref{coef: 1 in 1 hard}) it is sufficient to show that there exists an injection
$$\psi: \mathcal{T}_{(3,2,1^{n-5})} \to \mathcal{T}_{(2^3, 1^{n-6})}$$
with a constant weight shift. The details of $\psi$ are illustrated in Figure \ref{fig: 1 in 1 hard}. As seen easily this injection adds exactly one weight, since only the $\{1,2\}$ inversion is added. Furthermore for any $T \in \psi(\mathcal{T}_{(3,2,1^{n-5})})$ the pre-image can easily be recovered by moving the 1 in front of the 3. Therefore $\psi$ is injective.

\begin{figure}
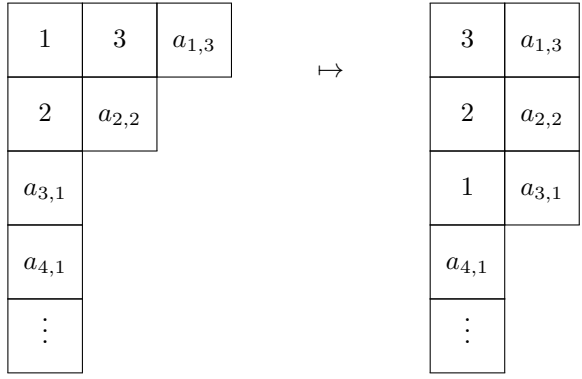

    \centering
    \begin{tabular}{c c c}
            \ytableausetup{mathmode, boxsize=2.75em}
             \begin{ytableau}
                1 & 3 & a_{1,3}\\
                2 & a_{2,2}\\
                a_{3,1}\\
                a_{4,1}\\
                \vdots
            \end{ytableau}&
            \qquad $\mapsto$ &
            \qquad
            \begin{ytableau}
            3 & a_{1,3}\\
            2 & a_{2,2}\\
            1 & a_{3,1}\\
            a_{4,1}\\
            \vdots
        \end{ytableau}
             
        \end{tabular}
    \caption{The map $\psi$}
    \label{fig: 1 in 1 hard}
\end{figure}
Hence $\widetilde{X}_G(t)$ is $e$-positive.
\end{proof}

Conjecture \ref{theo: 2,m2,m3} can even be further extended by combining it with already known $e$-positivity results summarized by Remark \ref{rem: known}.
With all this preparation we derive the following from Conjecture \ref{theo: 2,m2,m3}.
\begin{conj}\label{theo: m1_m2_m3,n,...,n}
Let $\mathbf{m}=(m_1,m_2,m_3,n,\ldots,n)$ and $G$ be the incomparability graph of $P(\mathbf{m})$. Then $\widetilde{X}_G(t)$ is $e$-positive.
\end{conj}

\renewcommand*{\proofname}{Justification}
\begin{proof}
Let $G$ denote the incomparability graph of $P(\mathbf{m})$. The sequence $\mathbf{m}=(m_1,m_2,m_3,n,\ldots,n)$ can be split into three different cases in respect to the value of $m_1$:
\begin{itemize} 
    \item $\mathbf{m_1=1}$: Then 1 is comparable with every element in $[n]$ and therefore contributes a disconnected vertex in the incomparability graph of $P(\mathbf{m})$. Since the chromatic (quasi)symmetric function of a single vertex is given by $e_1$, it is $e$-positive. The other part of the incomparability graph is isomorphic to the incomparability graph $H$ induced by the sequence $(m_2,m_3,n-1,\ldots,n-1)$, this is known to be $e$-positive by Remark \ref{rem: known}. So by equation (\ref{eq: disjoint quasi}) 
    $$\widetilde{X}_G(t)=e_1 \cdot \widetilde{X}_H(t)$$
    and is therefore $e$-positive.
    \item $\mathbf{m_1=2}$: This case is proven to be $e$-positive by Conjecture \ref{theo: 2,m2,m3}.
    \item $\mathbf{m_1\geq 3}$: This case can be reduced to $(r,m_2,m_3,\ldots,m_r,n,\ldots,n)$ by setting $m_i=n \: \mathrm{for} \: 4 \leq i \leq r$ and so it is $e$-positive.
\end{itemize}
\end{proof}
\section{Looking Forward}\label{future}
Investigations on simple cases show that any explicit weight preserving injection in the incomplete case would be quite pathological but the construction of such an injection certainly should be possible. The work of Shareshian and Wachs (\cite{SharWa}) and Cho and Huh (\cite{ChoHuh}) establishes that unit interval graphs on $n$ vertices with bodies of size $n-2$ are $e$-positive. This is the first step in extending that result to unit interval graphs with bodys of size $n-3$. The remaining cases are, up to isomorphism, induced by sequences of the form
$$(m_1,m_2,n-1,\dots,n-1,n,\dots,n),$$
where $1\le m_1\le m_2\le n-1$. More simply put they are precisely those induced by sequences with $n-1$ as the third element. An immediate observation is that this case has the same permissible tableaux shapes as Conjecture \ref{theo: 2,m2,m3}, so it would have the same set of injections. This suggests it might be possible in some cases to use similar or even identical injections as those developed here, for instance the injection used for $E_{(n-2,2)}$ works again without modification.
\section{Acknowledgements}

The authors thank Owen Merkel for suggesting the problem $e$-positivity of generalized nets, and Ch\'inh Ho\`{a}ng for suggesting the problem of $e$-positivity of claw-free, $P_4$-sparse graphs. The authors also thank Leo Tenenbaum for providing code used towards this project. This work was supported by the Canadian Tri-Council Research Support Fund. The author A.M.F. was supported by an NSERC Discovery Grant.  This research was conducted at the Fields Institute, Toronto, Canada as part of the 2018 Fields Undergraduate Summer Research Program and was funded by that program.

The authors thank the referees whose helpful suggestions have improved the paper.

\end{document}